\documentclass[10pt]{amsart}

\usepackage{amsmath,xspace,amssymb,mathrsfs}
\usepackage{color}

\input xy
\xyoption{all}
\xyoption{2cell}
\UseAllTwocells
\CompileMatrices

\newcommand{\Spec}{\operatorname{Spec}}
\renewcommand{\phi}{\varphi}

\newcommand{\Ker}{\operatorname{Ker}}

\newcommand{\Ima}{\operatorname{Im}}

\newcommand{\Min}{\operatorname{Min}}

\newtheorem{proposition}{Proposition}[section]
\newtheorem{lemma}[proposition]{Lemma}

\newtheorem{corollary}[proposition]{Corollary}
\newtheorem{theorem}[proposition]{Theorem}

\theoremstyle{definition}

\newtheorem{remark}[proposition]{Remark}

\usepackage{etoolbox}
\makeatletter
\patchcmd{\@settitle}{\uppercasenonmath\@title}{}{}{}
\patchcmd{\@setauthors}{\MakeUppercase}{}{}{}
\makeatother

\begin{document}

\title[Infinite Zariski closure]{Spectral closures of an infinite subset of the prime spectrum}

\author[A. Tarizadeh]{Abolfazl Tarizadeh}
\address{Department of Mathematics, Faculty of Basic Sciences, University of Maragheh \\ P.O. Box 55136-553, Maragheh, Iran.}
\email{ebulfez1978@gmail.com \\
atarizadeh@maragheh.ac.ir}
\date{}

\begin{abstract}
In the literature, there is no known general method (formula) to compute the Zariski closure of an ``infinite'' subset of the prime spectrum. This problem indeed deals with the prime ideals of an infinite direct product of nonzero commutative rings that are very complicated to understand (the structure of most of them is unknown). 
In this article, by appealing to the patch closure and using laying over minimal prime technique, we overcome the above obstacle and then obtain new and quite useful results for computing the Zariski and flat closures of an infinite subset of the prime spectrum.
\end{abstract}

\maketitle

\section{Introduction}

If $\mathfrak{p}$ is a prime ideal of a commutative ring $R$, then the Zariski closure of $\{\mathfrak{p}\}$ in $\Spec(R)$ equals to $\Ima\pi^{\ast}=V(\mathfrak{p})$ where $\pi: R\rightarrow R/\mathfrak{p}$ is the canonical ring map. We know that the structure of all (prime) ideals of any finite direct product ring $S=R_{1}\times\cdots\times R_{n}$ can be precisely described in terms of the (prime) ideals of the factor rings $R_{k}$.   
In fact, it can be easily seen that the ideals of $S$ are precisely of the form $I_{1}\times\cdots\times I_{n}$ where each $I_{k}$ is an ideal of $R_{k}$, and the prime ideals of $S$ are precisely of the form $\pi_{k}^{-1}(\mathfrak{p})$ where $\pi_{k}:S\rightarrow R_{k}$ is the projection map and $\mathfrak{p}$ is a prime ideal of $R_{k}$ for some $k$ (note that $\pi_{k}^{-1}(\mathfrak{p})=I_{1}\times\cdots\times I_{n}$ with $I_{k}=\mathfrak{p}$ and $I_{i}=R_{i}$ for all $i\neq k$). Hence, if $E=\{\mathfrak{p}_{1},\ldots,\mathfrak{p}_{n}\}$ is a finite set of prime ideals of a ring $R$ then it can be easily seen that the Zariski closure of $E$ in $\Spec(R)$ equals $\Ima\pi^{\ast}=\bigcup\limits_{i=1}^{n}V(\mathfrak{p}_{i})=
V(\bigcap\limits_{k=1}^{n}\mathfrak{p}_{k})$ where $\pi:R\rightarrow R/\mathfrak{p}_{1}\times\cdots\times R/\mathfrak{p}_{n}$ is the canonical ring map. Then quite inspired by this simple observation, we prove a similar result for any ``infinite'' subset $E$ of $\Spec(R)$. But before stating the theorem, we must note that the situation in the infinite case, that is, in the infinite direct product ring $\prod\limits_{\mathfrak{p}\in E}R/\mathfrak{p}$, is extremely complicated. Because the structure of all prime ideals of an infinite direct product of nonzero commutative rings is unknown (it is still unclear how the form of every prime ideal in an infinite direct product ring looks like). In addition to this, infinite direct products of commutative rings (even fields) have a very huge number of prime ideals. For example, the number of maximal ideals of the direct product ring  $\prod\limits_{n\geqslant1}\mathbb{Z}/p^{n}\mathbb{Z}$ with $p$ a prime number is uncountable and equals $2^{\mathfrak{c}}$ where $\mathfrak{c}$ is the cardinality of the continuum (the Krull dimension of this ring is also infinite, although the prime spectrum of $\mathbb{Z}/p^{n}\mathbb{Z}$ is singleton and hence its Krull dimension is zero for all $n$).
In spite of this intricacy, by appealing to the patch closure and using laying over minimal prime technique, we succeeded to show that for any commutative ring $R$ if $E$ is an infinite subset of $\Spec(R)$ then $\Ima\pi^{\ast}\subseteq\overline{E}$ where $\overline{E}$ denotes the Zariski closure of $E$ in $\Spec(R)$ and $\pi:R\rightarrow\prod\limits_{\mathfrak{p}\in E}R/\mathfrak{p}$ is the canonical ring map. We also show with an example that the inclusion $\Ima\pi^{\ast}\subseteq\overline{E}$ can be strict. Hence, our theorem (Theorem \ref{Theorem 8 Z closure}) is the most general result that allows us to approximate the Zariski closure of an infinite subset of the prime spectrum as accurately as possible. 

Next we investigate the dual version of the above result, namely computing the flat closure of an infinite subset of the prime spectrum. First recall that for any commutative ring $R$ there exists a (unique) topology over $\Spec(R)$ such that the collection of $V(a)$ with $a \in R$ forms a sub-basis for its opens. Hence, the collection of $V(I)$ with $I$ a finitely generated ideal of $R$ forms a basis for the opens of this topology. This topology is called the flat (or, inverse) topology which is the dual of the Zariski topology. 
If $\mathfrak{p}$ is a prime ideal of a ring $R$, then it can be seen that the flat closure of $\{\mathfrak{p}\}$ in $\Spec(R)$ equals $\Lambda(\mathfrak{p}):=\Ima\pi^{\ast}=
\{\mathfrak{q}\in\Spec(R):\mathfrak{q}
\subseteq\mathfrak{p}\}$ where $\pi: R\rightarrow R_{\mathfrak{p}}$ is the canonical ring map. Similarly above, if $E=\{\mathfrak{p}_{1},\ldots,\mathfrak{p}_{n}\}$ is a finite set of prime ideals of $R$ then  
the flat closure of $E$ in $\Spec(R)$ equals $\Ima\pi^{\ast}=
\bigcup\limits_{i=1}^{n}\Lambda(\mathfrak{p}_{i})$ where $\pi:R\rightarrow R_{\mathfrak{p}_{1}}\times\cdots\times R_{\mathfrak{p}_{n}}$ is the canonical ring map. But if $E$ is an infinite subset of $\Spec(R)$ then in order to compute the flat closure of $E$ in $\Spec(R)$ we have to deal, similarly above, with the infinite direct product ring $\prod\limits_{\mathfrak{p}\in E}R_{\mathfrak{p}}$ whose structure of most of its prime ideals (as stated in the above) is unknown. But by passing to the patch closure, we succeeded to show that  $\Ima\pi^{\ast}\subseteq c\ell(E)$ where $c\ell(E)$ denotes the flat closure of $E$ in $\Spec(R)$ and $\pi:R\rightarrow\prod\limits_{\mathfrak{p}\in E}R_{\mathfrak{p}}$ is the canonical ring map. We also show with an example that the inclusion $\Ima\pi^{\ast}\subseteq c\ell(E)$ can be strict. In the subsequent results (Theorems \ref{Theorem 11 ii}, \ref{Theorem 13-TH} and Corollaries \ref{Lemma i1}, \ref{Corollary 4 fd}) we obtain more observations on the Zariski and flat closures of some infinite subsets where the rings involved are certain rings. These results are also supplemented by Theorem \ref{Theorem supplement}, which provides an interesting example of such rings.   

\section{Preliminaries}

In this article, all rings are assumed to be commutative. If $r$ is an element of a ring $R$, then $D(r)=\{\mathfrak{p}\in\Spec(R): r\notin\mathfrak{p}\}$ and $V(r)=\Spec(R)\setminus D(r)$. Every ring map $\phi:R\rightarrow R'$ induces a map $\phi^{\ast}:\Spec(R')\rightarrow\Spec(R)$
between the corresponding prime spectra which is given by $\mathfrak{p}\mapsto\phi^{-1}(\mathfrak{p})$. 

Recall that for any ring $R$, there exists a (unique) topology over $\Spec(R)$ such that the collection of $D(a) \cap V(b)$ with $a,b \in R$ forms a sub-basis for its opens. It follows that the collection of $D(a) \cap V(I)$ with $a \in R$ and $I$ a finitely generated ideal of $R$ forms a basis for the opens of this topology. This topology is called the patch (or, constructible) topology. It can be seen that the patch closed subsets of $\Spec(R)$ are precisely of the form $\Ima\phi^{\ast}$ where $\phi:R\rightarrow A$ is a ring map. For further information on the patch and flat topologies see e.g. \cite{Tarizadeh 4}, \cite{Tarizadeh 3} and \cite{Tarizadeh 2}.

Let $R$ be a ring and $E$ a subset of $\Spec(R)$. We say that $E$ is stable under specialization if whenever $\mathfrak{p}\in E$ and $\mathfrak{q}$ is a prime ideal of $R$ with $\mathfrak{p}\subseteq\mathfrak{q}$, then $\mathfrak{q}\in E$. Dually, we say that $E$ is stable under generalization if whenever $\mathfrak{p}\in E$ and $\mathfrak{q}$ is a prime ideal of $R$ with $\mathfrak{q}\subseteq\mathfrak{p}$, then $\mathfrak{q}\in E$. It can be seen that a patch closed subset $E$ of $\Spec(R)$ is Zariski closed if and only if $E$ is stable under specialization. Dually, a patch closed subset $E$ of $\Spec(R)$ is flat closed if and only if $E$ is stable under generalization. 

Let $(R_{i})$ be a family of nonzero rings indexed by a set $S$ and let $R=\prod\limits_{i\in S}R_{i}$ be their direct product ring. If $\mathfrak{p}$ is a prime ideal of $R_k$ for some $k\in S$, then we call $\pi^{-1}_{k}(\mathfrak{p})$ a \emph{tame prime of $R$} where $\pi_{k}:R\rightarrow R_{k}$ is the projection map.
By a \emph{wild prime of $R$} we mean a prime ideal of $R$ which is not tame, i.e., it is not of the form $\pi^{-1}_{k}(\mathfrak{p})$. For each $k\in S$, we call the element $e_{k}:=(\delta_{i,k})_{i\in S}$ the $k$-th unit idempotent of $R=\prod\limits_{i\in S}R_{i}$ where $\delta_{i,k}$ is the Kronecker delta. The direct sum ideal $I=\bigoplus\limits_{i\in S} R_{i}$ of $R$ is generated by the idempotents $e_{i}$, and it can be seen that a prime ideal $P$ of $R$ is a wild prime if and only if $P\in V(I)$. 
By a minimal tame prime of $R$ we mean a tame prime of $R$ which is also a minimal prime of $R$. It can be easily seen that the minimal tame primes of $R$ are precisely of the form $\pi^{-1}_{k}(\mathfrak{p})$ where $\mathfrak{p}$ is a minimal prime of $R_{k}$ for some $k$. For more information on the tame and wild primes we refer the interested reader to \cite{Tarizadeh-Shir}. 

\begin{remark}\label{Remark iv lying over} (Laying over minimal prime) Remember that if $\phi:A\rightarrow B$ is an injective ring map and $\mathfrak{p}$ is a minimal prime ideal of $A$, then there exists a prime ideal $\mathfrak{q}$ of $B$ which laying over $\mathfrak{p}$, i.e., $\mathfrak{p}=\phi^{-1}(\mathfrak{q})$. Indeed, consider the following commutative (pushout) diagram: $$\xymatrix{
A\ar[r]^{\phi}\ar[d]^{\pi} &B\ar[d]^{}\\A_{\mathfrak{p}}\ar[r]^{}&
A_{\mathfrak{p}}\otimes_{A}B}$$ where $\pi$ and the unnamed arrows are the canonical maps. Clearly $A_{\mathfrak{p}}\otimes_{A}B\simeq B_{\mathfrak{p}}$ is a nonzero ring, since $\phi$ is injective. So it has a prime ideal $P$. The contraction of $P$ under the canonical ring map $A_{\mathfrak{p}}\rightarrow A_{\mathfrak{p}}\otimes_{A}B$ equals $\mathfrak{p}A_{\mathfrak{p}}$, because $\mathfrak{p}$ is a minimal prime of $A$. Let $\mathfrak{q}$ be the contraction of $P$ under the canonical ring map $B\rightarrow A_{\mathfrak{p}}\otimes_{A}B$ which is a prime ideal of $B$. Now we have $\phi^{-1}(\mathfrak{q})=
\pi^{-1}(\mathfrak{p}A_{\mathfrak{p}})=\mathfrak{p}$.
\end{remark}

The nil-radical of a ring $R$ is denoted by $\mathfrak{N}(R)$, and the Jacobson radical of $R$ is denoted by $\mathfrak{J}(R)$. By a von-Neumann regular ring we mean a ring $R$ such that for each $a\in R$ there exists some $b\in R$ such that $a(1-ab)=0$. It can be easily seen that every von-Neumann regular ring is a reduced zero-dimensional ring. It is also easy to check that every direct product of von-Neumann regular rings is a von-Neumann regular ring. In particular, every direct product of fields is a von-Neumann regular ring, and hence it is a (reduced) zero-dimensional ring. 

\section{Infinite Zariski and flat closures}

We begin with the following result which is already known for finite products. By $T(R)$ we mean the total ring of fractions of $R$.

\begin{lemma}\label{Lemma T-S} For $R=\prod\limits_{k\in S}R_{k}$, we have the canonical isomorphism of rings $T(R)\simeq\prod\limits_{k\in S}T(R_{k})$.
\end{lemma}

\begin{proof} It is straightforward. 
\end{proof}

If $\mathfrak{p}$ is a prime ideal of a ring $R$, then $\kappa(\mathfrak{p})$ denotes the residue field of $R$ at $\mathfrak{p}$, i.e., $\kappa(\mathfrak{p})
=R_{\mathfrak{p}}/\mathfrak{p}R_{\mathfrak{p}}$.

\begin{corollary}\label{Coro I-bir} If $R$ is a ring and $E$ is a subset of $\Spec(R)$, then the total ring of fractions of $\prod\limits_{\mathfrak{p}\in E}R/\mathfrak{p}$ is canonically isomorphic to $\prod\limits_{\mathfrak{p}\in E}\kappa(\mathfrak{p})$.
\end{corollary}

\begin{proof} It follows from Lemma \ref{Lemma T-S}. 
\end{proof}

\begin{lemma}\label{Theorem 12 mtd} For the ring $R=\prod\limits_{k\in S}R_{k}$ the following assertions are equivalent. \\
$\mathbf{(i)}$ $\mathfrak{N}(R)=\prod\limits_{k\in S}\mathfrak{N}(R_{k})$. \\
$\mathbf{(ii)}$ The set of minimal tame primes of $R$ is Zariski dense in $\Spec(R)$.
\end{lemma}

\begin{proof} (i)$\Rightarrow$(ii): It suffices to show that every nonempty Zariski open $U$ of $\Spec(R)$ meets the set of minimal tame primes of $R$. Since $U$ is nonempty, so there exists some $a=(a_{k})\in R$ such that $\emptyset\neq D(a)\subseteq U$. This shows that $a$ is not nilpotent. Thus by hypothesis, there exists some $k$ such that $a_{k}\notin\mathfrak{N}(R_{k})$. So $a_{k}\notin\mathfrak{p}$ for some minimal prime ideal $\mathfrak{p}$ of $R_{k}$. Then $\pi_{k}^{-1}(\mathfrak{p})$ is a minimal tame prime of $R$ with
$\pi_{k}^{-1}(\mathfrak{p})\in D(a)$. \\
(ii)$\Rightarrow$(i): It is clear that $\mathfrak{N}(R)\subseteq\prod\limits_{k\in S}\mathfrak{N}(R_{k})$. To prove the reverse inclusion, take $a=(a_{k})\in\prod\limits_{k\in S}\mathfrak{N}(R_{k})$.
If $a$ is not nilpotent, then $D(a)$ is nonempty. Then by hypothesis, there is a minimal tame prime $\pi_{k}^{-1}(\mathfrak{p})$ of $R$ such that $\pi_{k}^{-1}(\mathfrak{p})\in D(a)$ where $\mathfrak{p}$ is a (minimal) prime ideal of $R_{k}$ for some $k$. It follows that $a_{k}\notin\mathfrak{p}$ which is a contradiction, because $a_{k}$ is nilpotent.
\end{proof}

By $\gamma(E)$ we mean the patch closure of $E$ in $\Spec(R)$. In the following theorem, we give an alternative proof to our result \cite[Theorem 3.1]{Tarizadeh 2}. We will use this key result in the sequel. 

\begin{theorem}\label{Lemma 2ii} Let $R$ be a ring and $E$ a subset of $\Spec(R)$. Then $\gamma(E)=\Ima\pi^{\ast}$ where $\pi:R\rightarrow\prod\limits_{\mathfrak{p}\in E}\kappa(\mathfrak{p})$ is the canonical ring map.
\end{theorem}

\begin{proof} Clearly $E\subseteq\Ima\pi^{\ast}$. This yields that $\gamma(E)\subseteq\Ima(\pi^{\ast})$, because it is well known that the patch closed subsets of $\Spec(R)$ are precisely of the form $\Ima(\phi^{\ast})$ where $\phi:R\rightarrow A$ is a ring map. To see the reverse inclusion, we act as follows. We know that $T:=\prod\limits_{\mathfrak{p}\in E}\kappa(\mathfrak{p})$ is a von-Neumann regular ring, and so it is a reduced zero-dimensional ring. The Zariski and patch topologies over the prime spectrum of a zero-dimensional ring are the same. Thus by Lemma \ref{Theorem 12 mtd}, $\gamma(E')=\Spec(T)$ where $E':=\{\pi^{-1}_{\mathfrak{p}}(0):\mathfrak{p}\in E\}$ denotes the set of tame primes of $T$ and $\pi_{\mathfrak{p}}:T\rightarrow\kappa(\mathfrak{p})$ is the projection map. It is obvious that the map $\pi^{\ast}:\Spec(T)\rightarrow\Spec(R)$ is continuous with respect to the patch topology, and $\pi^{\ast}(E')=E$. It is also easy to see that if $f: X\rightarrow Y$ is a continuous map of topological spaces and $E\subseteq X$, then $f(\overline{E})\subseteq\overline{f(E)}$. Using these observations, then we have $\Ima\pi^{\ast}=\pi^{\ast}\big(\Spec(T)\big)=
\pi^{\ast}\big(\gamma(E')\big)\subseteq
\gamma\big(\pi^{\ast}(E')\big)=\gamma(E)$.
\end{proof}

The following theorem is one of the main results of this article. If $E$ is a subset of $\Spec(R)$ then $\overline{E}$ denotes the Zariski closure of $E$ in $\Spec(R)$.

\begin{theorem}\label{Theorem 8 Z closure} Let $R$ be a ring and $E$ a subset of $\Spec(R)$. Then $\Ima\pi^{\ast}\subseteq\overline{E}$ where $\pi:R\rightarrow\prod\limits_{\mathfrak{p}\in E}R/\mathfrak{p}$ is the canonical ring map.
\end{theorem}

\begin{proof} If $P$ is a prime ideal of $S:=\prod\limits_{\mathfrak{p}\in E}R/\mathfrak{p}$ then there exists a minimal prime ideal $P'$ of $S$ with $P'\subseteq P$. We know that $\kappa(\mathfrak{p})$ is the field of fractions of the integral domain $R/\mathfrak{p}$, and so the ring map $\phi:S\rightarrow\prod\limits_{\mathfrak{p}\in E}\kappa(\mathfrak{p})$ induced by the canonical injective ring maps $R/\mathfrak{p}\rightarrow\kappa(\mathfrak{p})$ is injective. In fact, by Lemma \ref{Lemma T-S} (or by Corollary \ref{Coro I-bir}), $T$ is the total ring of fractions of $S$ and $\phi:S\rightarrow T$ is the canonical ring map. 
Then by Remark \ref{Remark iv lying over}, there exists a prime ideal $P''$ of $T:=\prod\limits_{\mathfrak{p}\in E}\kappa(\mathfrak{p})$ which laying over $P'$, i.e., $P'=\phi^{-1}(P'')$. But $\phi\circ\pi:R\rightarrow T$ is the canonical ring map. Using this and Theorem \ref{Lemma 2ii}, we get that
$\pi^{-1}(P')\in\gamma(E)$. But $\gamma(E)\subseteq\overline{E}$, because the patch topology is finer than the Zariski topology. It follows that $\pi^{-1}(P)\in\overline{E}$, since every Zariski closed subset of the prime spectrum is stable under specialization.
Hence, $\Ima\pi^{\ast}\subseteq\overline{E}$.
\end{proof}

In Remark \ref{Remark v5}, we will observe that the inclusion $\Ima\pi^{\ast}\subseteq\overline{E}$ can be strict. \\

In what follows, $c\ell(E)$ denotes the closure of $E$ in $\Spec(R)$ with respect to the flat topology. 
In the following result, $U(R)=\{r\in R:\exists r'\in R, rr'=1\}$ denotes the group of units of $R$.

\begin{lemma}\label{Prop 1i1} For a given ring $R$ we have: \\
$\mathbf{(i)}$ Every infinite subset of $\Spec(R)$ is Zariski dense if and only if $V(a)$ is a finite set for all $a\in R\setminus\mathfrak{N}(R)$. \\
$\mathbf{(ii)}$ Every infinite subset of $\Spec(R)$ is flat dense if and only if $D(a)$ is a finite set for all $a\in R\setminus U(R)$.
\end{lemma}

\begin{proof} (i): Assume every infinite subset of $\Spec(R)$ is Zariski dense. Suppose $E:=V(a)$ is infinite for some $a\in R\setminus\mathfrak{N}(R)$. Then $V(a)=\overline{E}=\Spec(R)$. This shows that $a$ is nilpotent which is a contradiction. Conversely, let $E$ be an infinite subset of $\Spec(R)$. It suffices to show that each nonempty Zariski open $U$ of $\Spec(R)$ meets $E$. But there exists some $a\in R$ with $D(a)\subseteq U$ such that $D(a)$ is nonempty. This shows that $a$ is not nilpotent. If $D(a)\cap E=\emptyset$ then $E\subseteq V(a)$ which is impossible, since $E$ is infinite. \\
(ii): Assume every infinite subset of $\Spec(R)$ is flat dense. Suppose $E:=D(a)$ is infinite for some $a\in R\setminus U(R)$. Then $D(a)=c\ell(E)=\Spec(R)$. This shows that $a$ is invertible in $R$ which is a contradiction. Conversely, let $E$ be an infinite subset of $\Spec(R)$. It will be enough to show that each nonempty flat open $U$ of $\Spec(R)$ meets $E$. Since $U$ is nonempty, so there exists a finitely generated ideal $I=(a_{1},\ldots,a_{n})$ of $R$ such that $\emptyset\neq V(I)\subseteq U$. It follows that $a_{i}\in R\setminus U(R)$ for all $i$. If $V(I)\cap E=\emptyset$ then $E\subseteq V(I)^{c}=\big(\bigcap\limits_{i=1}^{n}V(a_{i})\big)^{c}=
\bigcup\limits_{i=1}^{n}D(a_{i})$ which is impossible, because $E$ is infinite.
\end{proof}

\begin{corollary}\label{Lemma i1} If $R$ is a PID or more generally a Dedekind domain, then every infinite subset of $\Spec(R)$ is Zariski dense.
\end{corollary}

\begin{proof} By Lemma \ref{Prop 1i1}(i), it suffices to show that $V(a)$ is a finite set for all nonzero $a\in R$. We may assume $a$ is a nonunit of $R$ (because if $a$ is invertible in $R$ then $V(a)=\emptyset$). There exists a finite set $\{\mathfrak{p}_{1},\ldots,\mathfrak{p}_{n}\}$ (with cardinality $\leqslant n$) of prime ideals of $R$ such that $Ra=\prod\limits_{i=1}^{n}\mathfrak{p}_{i}$, because in a Dedekind domain every proper ideal can be written as a finite product of prime ideals. Using this and the fact that in a Dedekind domain every nonzero prime ideal is maximal, we get that $V(a)=\{\mathfrak{p}_{1},\ldots,\mathfrak{p}_{n}\}$.
\end{proof}

\begin{remark}\label{Remark v5} Let $R$ be a ring and $E$ a subset of $\Spec(R)$. Then it is clear that $E\subseteq\Ima\pi^{\ast}$ where $\pi:R\rightarrow\prod\limits_{\mathfrak{p}\in E}R/\mathfrak{p}$ is the canonical ring map. Using Theorem \ref{Theorem 8 Z closure}, then $\overline{E}=\Ima\pi^{\ast}$ if and only if $\Ima\pi^{\ast}$ is stable under specialization (e.g. the map $\pi$ has the going-up property). If $E$ is a finite set then $\overline{E}=\Ima\pi^{\ast}=\bigcup\limits_{\mathfrak{p}\in E}V(\mathfrak{p})$. But here we show with an example that the inclusion $\Ima\pi^{\ast}\subseteq\overline{E}$ in Theorem \ref{Theorem 8 Z closure} can be strict for some infinite subsets. To achieve this, let $S$ be a proper infinite set of prime numbers and take $E=\{p\mathbb{Z}: p\in S\}$ which is Zariski dense in $\Spec(\mathbb{Z})$ by Corollary \ref{Lemma i1}. The image of each prime number $q$ with $q\notin S$ under the canonical ring map $\pi:\mathbb{Z}\rightarrow\prod\limits_{p\in S}\mathbb{Z}/p\mathbb{Z}$ is invertible. This shows that $q\mathbb{Z}\notin\Ima\pi^{\ast}$. Hence, the set $\Ima\pi^{\ast}$ is strictly contained in $\overline{E}=\Spec(\mathbb{Z})$. 
\end{remark}

The following result can be viewed as the dual of Theorem \ref{Theorem 8 Z closure}.

\begin{theorem}\label{Theorem 9 F closure} Let $R$ be a ring and $E$ a subset of $\Spec(R)$. Then $\Ima\pi^{\ast}\subseteq c\ell(E)$ where $\pi:R\rightarrow\prod\limits_{\mathfrak{p}\in E}R_{\mathfrak{p}}$ is the canonical ring map.
\end{theorem}

\begin{proof} If $P$ is a prime ideal of $S:=\prod\limits_{\mathfrak{p}\in E}R_{\mathfrak{p}}$ then $P\subseteq M$ for some maximal ideal $M$ of $S$. The ring map $\phi: S\rightarrow T:=\prod\limits_{\mathfrak{p}\in E}\kappa(\mathfrak{p})$ induced by the canonical ring maps $\phi_{\mathfrak{p}}:R_{\mathfrak{p}}
\rightarrow\kappa(\mathfrak{p})$ is surjective and $\Ker(\phi)=\prod\limits_{\mathfrak{p}\in E}\Ker(\phi_{\mathfrak{p}})=\prod\limits_{\mathfrak{p}\in E}\mathfrak{p}R_{\mathfrak{p}}=\prod\limits_{\mathfrak{p}\in E}\mathfrak{J}(R_{\mathfrak{p}})=\mathfrak{J}(S)\subseteq M$. So there exists a maximal ideal $M'$ of $T$ which laying over $M$, i.e., $\phi^{-1}(M')=M$. Clearly $\phi\circ\pi:R\rightarrow T$ is the canonical ring map. Using this and Theorem \ref{Lemma 2ii}, we obtain that
$\pi^{-1}(M)\in\gamma(E)$. But $\gamma(E)\subseteq c\ell(E)$, because the patch topology is finer than the flat topology. It follows that $\pi^{-1}(P)\in c\ell(E)$, since every flat closed subset of the prime spectrum is stable under generalization. Hence, $\Ima\pi^{\ast}\subseteq c\ell(E)$. 
\end{proof}

In Remark \ref{Remark 6 fc}, we will observe that the inclusion in Theorem \ref{Theorem 9 F closure} can be strict.
Recall that by a C.P. (or, compactly packed) ring we mean a ring $R$ which satisfies infinite prime avoidance: If an ideal $I$ of $R$ is contained in the union of a family $(\mathfrak{p}_{i})$ of prime ideals of $R$, then $I\subseteq\mathfrak{p}_{k}$ for some $k$. 
Dually, by a P.Z. (or, properly zipped) ring we mean a ring $R$ which satisfies infinite prime absorbance: If a prime ideal $\mathfrak{p}$ of $R$ contains the intersection of a family $(\mathfrak{p}_{i})$ of prime ideals of $R$, then $\mathfrak{p}_{k}\subseteq\mathfrak{p}$ for some $k$. It is well known that a Dedekind domain is C.P. if and only if its ideal class group has torsion. For more information on C.P. rings, P.Z. rings (and the ideal avoidance property) we refer the interested reader to the articles \cite{Tarizadeh-Chen 2}, \cite{Reis-Visw} and \cite{Tarizadeh-Chen}.

Further analysis of the example given in Remark \ref{Remark v5} leads us to the following result.

\begin{theorem}\label{Theorem 11 ii} Let $R$ be a Dedekind domain with torsion ideal class group and $E$ an infinite set of maximal ideals of $R$. Then the following assertions hold. \\
$\mathbf{(i)}$ $\Ima\pi^{\ast}=E\cup\{0\}$ and for every wild prime $P$ of $S:=\prod\limits_{\mathfrak{p}\in E}R/\mathfrak{p}$ we have $\pi^{-1}(P)=0$ where $\pi:R\rightarrow S$ is the canonical ring map. \\
$\mathbf{(ii)}$ $\Ima\pi^{\ast}=c\ell(E)=E\cup\{0\}$ and for every wild prime $P$ of $S:=\prod\limits_{\mathfrak{p}\in E}R_{\mathfrak{p}}$ we have $\pi^{-1}(P)=0$ where $\pi:R\rightarrow S$ is the canonical ring map.
\end{theorem}

\begin{proof} (i): First we show that $\Ima\pi^{\ast}\subseteq E\cup\{0\}$.
Take $\mathfrak{p}\in\Ima\pi^{\ast}$. Suppose $\mathfrak{p}\notin E\cup\{0\}$. By \cite[Theorem 2.2]{Reis-Visw}, $R$ has the infinite prime avoidance property. So we may choose some  $a\in\mathfrak{p}$ such that $a\notin\bigcup\limits_{\mathfrak{q}\in E}\mathfrak{q}$. There exists a prime ideal $\mathfrak{p}'$ of $S$ such that $\mathfrak{p}=\pi^{-1}(\mathfrak{p}')$. It follows that $\pi(a)=(a+\mathfrak{q})_{\mathfrak{q}\in E}\in\mathfrak{p}'$. Note that each $a+\mathfrak{q}$ is a nonzero element of the field $R/\mathfrak{q}$. This shows that $\pi(a)$ is invertible in $S$ which is a contradiction. Therefore $\Ima(\pi^{\ast})\subseteq E\cup\{0\}$. It is clear that $E\subseteq\Ima\pi^{\ast}$. Since $E$ is infinite, the direct sum ideal $I=\bigoplus\limits_{\mathfrak{p}\in E}R/\mathfrak{p}$ is a proper ideal of $S$ (i.e. $I\neq S$). 
Thus $I$ is contained in a prime ideal of $S$ which is a wild prime. Now, to conclude the assertion, it will be enough to show that if $P$ is a wild prime of $S$ then $\pi^{-1}(P)=0$. If $\pi^{-1}(P)$ is not zero, then $\pi^{-1}(P)=\mathfrak{q}$ for some $\mathfrak{q}\in E$. We may choose some $a\in\mathfrak{q}$ such that $a\notin\bigcup\limits_{\substack{\mathfrak{p}\in E,\\\mathfrak{p}\neq\mathfrak{q}}}\mathfrak{p}$. This yields that $\pi(a)=(a+\mathfrak{p})_{\mathfrak{p}\in E}\in P$. If $\mathfrak{p}\neq\mathfrak{q}$ then $a+\mathfrak{p}$ is a nonzero element of the field $R/\mathfrak{p}$. Then consider the sequence $x=(x_{\mathfrak{p}})$ in $S$ where $x_{\mathfrak{q}}=1$ and $x_{\mathfrak{p}}=(a+\mathfrak{p})^{-1}$ for all $\mathfrak{p}\in E$ with $\mathfrak{p}\neq\mathfrak{q}$. Then
$1-e_{\mathfrak{q}}=x\pi(a)\in P$ which is a contradiction where $e_{\mathfrak{q}}$ is the $\mathfrak{q}$-th unit idempotent of $S$. \\
(ii): The inclusion $\Ima\pi^{\ast}\subseteq c\ell(E)$ follows from Theorem \ref{Theorem 9 F closure}. It is easy to check that $c\ell(E)\subseteq E\cup\{0\}$, because $R$ has the infinite prime avoidance property. It is obvious that $E\subseteq\Ima\pi^{\ast}$. Similarly above, since $E$ is infinite, so $S$ has wild primes. Thus to conclude the assertion, it suffices to show that if $P$ is a wild prime of $S$ then $\pi^{-1}(P)=0$. If $\pi^{-1}(P)\neq0$ then
$\pi^{-1}(P)=\mathfrak{q}$ for some $\mathfrak{q}\in E$. We may choose some $a\in\mathfrak{q}$ such that $a\notin\bigcup\limits_{\substack{\mathfrak{p}\in E,\\\mathfrak{p}\neq\mathfrak{q}}}\mathfrak{p}$. This yields that $\pi(a)=(a/1)_{\mathfrak{p}\in E}\in P$. Consider the sequence $x=(x_{\mathfrak{p}})$ in $S$ where $x_{\mathfrak{q}}=0$ and $x_{\mathfrak{p}}=1/a$ for all $\mathfrak{p}\in E$ with $\mathfrak{p}\neq\mathfrak{q}$. Then $1-e_{\mathfrak{q}}=x\pi(a)\in P$ which is a contradiction.
\end{proof}

Note that every PID satisfies the hypothesis of the above theorem. In fact, the ideal class group of a PID is zero, so it has torsion. It is also well known that every Abelian group is isomorphic to the ideal class group of a Dedekind domain (see \cite[Theorem 7]{Claborn}). Hence, Dedekind domains with torsion ideal class groups are ubiquitous. For instance, if $n\geqslant1$ then there exists a Dedekind domain whose class group is isomorphic to the torsion additive group $\mathbb{Z}_{n}=\mathbb{Z}/n\mathbb{Z}$ (in fact, if $R$ is a ring of nonzero characteristic, then the additive group of $R$ has torsion, and the exponent of this group is the characteristic of $R$). \\

The following result is the dual of Theorem \ref{Theorem 11 ii}.

\begin{theorem}\label{Theorem 13-TH} Let $(R,\mathfrak{m})$ be a reduced local P.Z. ring with $\dim(R)=1$ and $E$ an infinite set of minimal primes of $R$. Then the following assertions hold. \\
$\mathbf{(i)}$ $\Ima\pi^{\ast}=E\cup\{\mathfrak{m}\}$ and for every wild prime $P$ of $S:=\prod\limits_{\mathfrak{p}\in E}R_{\mathfrak{p}}$ we have $\pi^{-1}(P)=\mathfrak{m}$ where $\pi:R\rightarrow S$ is the canonical ring map. \\
$\mathbf{(ii)}$ $\Ima\pi^{\ast}=\overline{E}=E\cup\{\mathfrak{m}\}$ and for every wild prime $P$ of $S:=\prod\limits_{\mathfrak{p}\in E}R/\mathfrak{p}$ we have $\pi^{-1}(P)=\mathfrak{m}$ where $\pi:R\rightarrow S$ is the canonical ring map.
\end{theorem}

\begin{proof} (i): First we show that $\Ima\pi^{\ast}\subseteq E\cup\{\mathfrak{m}\}$. Let $P$ be a prime ideal of $S$ and setting $\mathfrak{q}:=\pi^{-1}(P)$. Suppose $\mathfrak{q}\notin E\cup\{\mathfrak{m}\}$.
Since $R$ is a $P.Z.$ ring, so there exists some $a\in\bigcap\limits_{\mathfrak{p}\in E}\mathfrak{p}$ such that $a\notin\mathfrak{q}$.
For each $\mathfrak{p}\in E$, $R_{\mathfrak{p}}$ is a reduced ring and so $a/1\in\mathfrak{p}R_{\mathfrak{p}}=0$. This means that $\pi(a)=(a/1)_{\mathfrak{p}\in E}$ is the zero sequence. But $\pi(a)\notin P$ which is a contradiction. Thus $\Ima\pi^{\ast}\subseteq E\cup\{\mathfrak{m}\}$. It is clear that $E\subseteq\Ima\pi^{\ast}$. Since $E$ is infinite, $S$ has wild primes. Hence, to conclude the assertion, it suffices to show that if $P$ is a wild prime of $S$ then $\pi^{-1}(P)=\mathfrak{m}$. If $\pi^{-1}(P)\neq\mathfrak{m}$ then $\pi^{-1}(P)=\mathfrak{q}$ for some $\mathfrak{q}\in E$. Since $R$ is a P.Z. ring, so we may choose some $a\in\bigcap\limits_{\substack{\mathfrak{p}\in E,\\\mathfrak{p}\neq\mathfrak{q}}}\mathfrak{p}$ such that $a\notin\mathfrak{q}$.
Thus $\pi(a)=(a/1)_{\mathfrak{q}\in E}\notin P$ and $a/1$ is invertible in $R_{\mathfrak{q}}$. Also, $a/1=0$ in $R_{\mathfrak{p}}$ for all $\mathfrak{p}\in E$ with $\mathfrak{p}\neq\mathfrak{q}$. Then consider the sequence $x=(x_{\mathfrak{q}})_{\mathfrak{q}\in E}$ in $S$ with $x_{\mathfrak{q}}=1/a$ and $x_{\mathfrak{p}}=1$ for all $\mathfrak{p}\in E$ with $\mathfrak{p}\neq\mathfrak{q}$. Thus $x$ is invertible in $S$ and so $x\notin P$. It follows that the $\mathfrak{q}$-th unit idempotent $e_{\mathfrak{q}}=x\pi(a)\notin P$ which is a contradiction. \\
(ii): The inclusion $\Ima\pi^{\ast}\subseteq\overline{E}$ follows from Theorem \ref{Theorem 8 Z closure}. Since $R$ is a P.Z. ring, so the inclusion $\overline{E}\subseteq E\cup\{\mathfrak{m}\}$ is also clear. To conclude the assertion, it suffices to show that if $P$ is a wild prime of $S$ then $\pi^{-1}(P)=\mathfrak{m}$. Suppose $\pi^{-1}(P)=\mathfrak{q}$ for some $\mathfrak{q}\in E$. We may choose some $a\in\bigcap\limits_{\substack{\mathfrak{p}\in E,\\\mathfrak{p}\neq\mathfrak{q}}}\mathfrak{p}$ such that $a\notin\mathfrak{q}$. Thus $\pi(a)=(a+\mathfrak{p})_{\mathfrak{p}\in E}\notin P$. Also, $1-e_{\mathfrak{q}}\notin P$. This yields that $0=(1-e_{\mathfrak{q}})\pi(a)\notin P$ which is a contradiction.
\end{proof}

As a sort of the dual of Corollary \ref{Lemma i1}, we have the following result.

\begin{corollary}\label{Corollary 4 fd} If $(R,\mathfrak{m})$ is a local P.Z. ring with $\dim(R)=1$, then every infinite subset of $\Spec(R)$ is flat dense.
\end{corollary}

\begin{proof} Let $E$ be an infinite subset of $\Spec(R)$.
If $\mathfrak{m}\in E$ then $c\ell(E)=\Spec(R)$, because $c\ell(E)$ is stable under the generalization. Now assume $\mathfrak{m}\notin E$.
Without loss of generality, we may assume $R$ is reduced, because the canonical ring map $R\rightarrow R/\mathfrak{N}$ induces a homeomorphism between the corresponding prime spectra (with respect to the flat topology) and $R/\mathfrak{N}$ is still a local P.Z. ring of Krull dimension one.
Then by Theorem \ref{Theorem 13-TH}(i), $\Ima\pi^{\ast}=E\cup\{\mathfrak{m}\}$ where $\pi:R\rightarrow\prod\limits_{\mathfrak{p}\in E}R_{\mathfrak{p}}$ is the canonical ring map. If $\mathfrak{p}\in E$ then $R_{\mathfrak{p}}$ is a field. Then by Theorem \ref{Lemma 2ii}
(or, by Theorem \ref{Theorem 9 F closure}), $\Ima(\pi^{\ast})\subseteq c\ell(E)$. Thus $\mathfrak{m}\in c\ell(E)$. But $c\ell(E)$ is stable under generalization, and hence $c\ell(E)=\Spec(R)$.
\end{proof}

In the next result we construct an explicit example of a  reduced local P.Z. ring with Krull dimension one and infinitely many minimal primes. Note that in the following result, the index set $S$ is nonempty.

\begin{theorem}\label{Theorem supplement} Let $T=K[x_{i}: i\in S]$ with $K$ a ring, $I=(x_{i}x_{k}:i,k\in S, i\neq k)$ and $I_{k}=(x_{i}: i\in S, i\neq k)$ for all $k\in S$. Then the following assertions hold. \\
$\mathbf{(i)}$ $I=\bigcap\limits_{k\in S}I_{k}$. In particular, $K$ is reduced if and only if $T/I$ is reduced. \\
$\mathbf{(ii)}$ $K$ is an integral domain if and only if $\Min(I)=\{I_{k}: k\in S\}$. \\
$\mathbf{(iii)}$ If $K$ is a Noetherian domain, then $\dim(T/I)=\dim(K)+1$. \\
$\mathbf{(iv)}$ If $S$ is infinite and $K$ is a field, then $R:=(T/I)_{\mathfrak{m}}$ is a reduced local P.Z. ring with $\dim(R)=1$ and infinitely many minimal primes where $\mathfrak{m}=(x_{i}: i\in S)$.
\end{theorem}

\begin{proof} (i): It is clear that $I\subseteq I_{k}$ for all $k\in S$. Thus $I\subseteq\bigcap\limits_{k\in S}I_{k}$.
Now take $f\in\bigcap\limits_{k\in S}I_{k}$. To see $f\in I$ it suffices to show that in the expansion of $f$ for every monomial which appears, at least two distinct variables with positive powers are involved. If not, then there exist $g(x_{d})\in K[x_{d}]$ (for some $d\in S$) and $h\in T$ such that $f=g(x_{d})+h$ and there is no monomial of the form $x^{n}_{d}$ (with $n\geqslant0$) in the expansion of $h$. It follows that $h\in I_{d}$ and so $g(x_{d})\in I_{d}$ which is a contradiction, since for each nonzero $r\in R$ and $n\geqslant0$, then $rx^{n}_{d}\notin I_{d}$. Hence, $g(x_{d})=0$. Therefore $I=\bigcap\limits_{k\in S}I_{k}$. If $K$ is reduced then $T/I_{k}\simeq K[x_{k}]$ is also reduced, and so $I_{k}$ is a radical ideal for all $k\in S$. Thus the ideal $I$ is also radical, because the intersection of every family of radical ideals is a radical ideal. This yields that $T/I$ is reduced. Conversely, if $T/I$ is reduced then $I$ is a radical ideal. Now if $a\in K$ is nilpotent then $a^{n}=0\in I$ for some $n\geqslant1$, and so $a\in\sqrt{I}=I$. It follows that $a=0$. \\
(ii): If $K$ is an integral domain, then $I_{k}$ is a prime ideal of $T$ for all $k$. If $\mathfrak{p}$ is a prime ideal of $T$ with $I\subseteq\mathfrak{p}\subseteq I_{k}$ for some $k$, then for each $i$ with $i\neq k$ we have $x_{i}x_{k}\in I$. Since $x_{k}\notin\mathfrak{p}$, so $x_{i}\in\mathfrak{p}$. Thus $\mathfrak{p}=I_{k}$. To see the reverse inclusion, take $\mathfrak{p}\in\Min(I)$. We may assume $|S|\geqslant2$. Then $\mathfrak{p}\neq(x_{k}: k\in S)$. Thus $x_{k}\notin\mathfrak{p}$ for some $k$. But for each $i$ with $i\neq k$ we have $x_{i}x_{k}\in I$ and so $x_{i}\in\mathfrak{p}$. This shows that $I_{k}\subseteq\mathfrak{p}$ and so $I_{k}=\mathfrak{p}$. Conversely, since $S$ is nonempty so we may choose some $k\in S$. By hypothesis, $I_{k}$ is a prime ideal of $T$ thus $T/I_{k}\simeq K[x_{k}]$ is an integral domain, and so every its subring, especially $K$, is an integral domain. \\
(iii): We know that the Krull dimension of any ring $R$ is the supremum of all $\dim(R/\mathfrak{p})$ where $\mathfrak{p}$ ranges over the set of minimal primes of $R$. Also, since $K$ is a Noetherian ring thus $T/I_{k}\simeq K[x_{k}]$ has the Krull dimension $\dim(K)+1$.
Now by using these observations and (ii), we get that $\dim(T/I)=\dim(K)+1$. \\
(iv): If $K$ is a field then by using (iii), we have $1\leqslant\dim(R)\leqslant\dim(T/I)=1$. Hence, the non-maximal prime ideals of the local ring $R\simeq T_{\mathfrak{m}}/IT_{\mathfrak{m}}$ are precisely the minimal primes $P_{k}:=I_{k}T_{\mathfrak{m}}/IT_{\mathfrak{m}}$ of $R$ with $k\in S$. Since $S$ is infinite, thus $R$ has infinitely many minimal primes.
By (i), $T/I$ and so every its localization, especially $R$, are reduced.
It remains to show that $R$ is a P.Z. ring. Suppose $\bigcap\limits_{k\in S'}P_{k}\subseteq P=P_{d}$ where $S'$ is a nonempty subset of $S$ and $d\in S$ (the case of ``$P=$ the maximal ideal'' is clear). It suffices to show that $d\in S'$. We have $\bigcap\limits_{k\in S'}P_{k}=(\bigcap\limits_{k\in S'}I_{k}T_{\mathfrak{m}})/IT_{\mathfrak{m}}$ and clearly $\mathfrak{p}T_{\mathfrak{m}}\subseteq\bigcap\limits_{k\in S'}I_{k}T_{\mathfrak{m}}$ where $\mathfrak{p}=(x_{i}: i\in S\setminus S')$. It follows that $\mathfrak{p}\subseteq I_{d}$. This yields that $d\in S'$ which completes the proof.
\end{proof}

In addition to the above result, Hochster's characterization of spectral spaces provides us further examples of (reduced) local P.Z. rings with Krull dimension one and infinitely many minimal primes. In fact, by applying \cite[Proposition 5.5]{Tarizadeh-Chen} to every (nonfield) Dedekind domain $A$ with torsion class group we obtain a reduced local P.Z. ring $R$  such that $\Spec(A)$ endowed with the Zariski topology is isomorphic to $\Spec(R)$ equipped with the flat topology, and this homeomorphism reverses prime orders. In particular, $\dim(R)=\dim(A)=1$. If moreover, $A$ has infinitely many maximal ideals, then the resulting ring $R$ will have infinitely many minimal primes. 

\begin{remark}\label{Remark 6 fc} Let $R$ be a ring and $E$ a subset of $\Spec(R)$. Then it is clear that $E\subseteq\Ima\pi^{\ast}$ where $\pi:R\rightarrow\prod\limits_{\mathfrak{p}\in E}R_{\mathfrak{p}}$ is the canonical ring map. Then using Theorem \ref{Theorem 9 F closure}, we have  $\Ima\pi^{\ast}=c\ell(E)$ if and only if $\Ima\pi^{\ast}$ is stable under generalization (e.g. $\pi$ has the going-down property). If $E$ is a finite set then $c\ell(E)=\Ima\pi^{\ast}=\bigcup\limits_{\mathfrak{p}\in E}\Lambda(\mathfrak{p})$. But here we show with an example that the inclusion $\Ima\pi^{\ast}\subseteq c\ell(E)$ can be strict for some infinite subsets. To that end, let $(R,\mathfrak{m})$ be a reduced local P.Z. ring with $\dim(R)=1$
and infinitely many minimal primes (see Theorem \ref{Theorem supplement}), and let $E$ be a proper infinite set of minimal primes of $R$. Then by Corollary \ref{Corollary 4 fd}, $E$ is flat dense in $\Spec(R)$. But by Theorem \ref{Theorem 13-TH}(i), $\Ima\pi^{\ast}=E\cup\{\mathfrak{m}\}$. Hence, the set $\Ima\pi^{\ast}$ is strictly contained in $c\ell(E)=\Spec(R)$. 
\end{remark}

\end{document}